\documentclass[11pt]{amsart}

\usepackage{amsfonts,amsmath,amssymb,amscd}

\usepackage{amsthm}
\usepackage{latexsym}
\usepackage[dvips]{graphicx}
\usepackage[dvips]{psfrag}
\usepackage[dvips]{color}

\copyrightinfo{2006}{American Mathematical Society}

\newtheorem{theorem}{Theorem}[section]
\newtheorem{lemma}[theorem]{Lemma}

\theoremstyle{definition}

\newcommand{\Nbd}{\operatorname{Nbd}}
\newcommand{\cl}{\operatorname{cl}}

\numberwithin{equation}{section}

\begin{document}

\title[Primitive disk complexes]{Primitive disk complexes for lens spaces}

\author{Sangbum Cho}\thanks{The first-named author is supported in part by Basic Science Research Program through the National Research Foundation of Korea(NRF) funded by the Ministry of Education, Science and Technology (2012006520).}

\address{Department of Mathematics Education, Hanyang University, Seoul 133-791, Korea}
\email{scho@hanyang.ac.kr}

\author{Yuya Koda}\thanks{The second-named author is supported in part by
Grant-in-Aid for Young Scientists (B) (No. 20525167), Japan Society for the Promotion of Science.}

\address{Mathematical Institute, Tohoku University, Sendai, 980-8578, Japan}
\email{koda@math.tohoku.ac.jp}


\date{\today}

\begin{abstract}
For a genus two Heegaard splitting of a lens space, the primitive disk complex is defined to be the full subcomplex of the disk complex for one of the handlebodies of the splitting spanned by all vertices of primitive disks.
In this work, we describe the complete combinatorial structure of the primitive disk complex for the genus two Heegaard splitting of each lens space. In particular, we find all lens spaces whose primitive disk complexes are contractible.
\end{abstract}

\maketitle

\section{Introduction}
\label{sec:intro}

Every closed orientable $3$-manifold can be decomposed into two handlebodies of the same genus, which is called a Heegaard splitting of the manifold.
The genus of the handlebodies is called the genus of the splitting.
The $3$-sphere admits a Heegaard splitting of each genus $g \geq 0$, and a lens space a Heegaard splitting of each genus $g \geq 1$.
Furthermore, they are known to have a unique Heegaard splitting up to isotopy for each genus.

There is a well known simplicial complex, called the disk complex, for a handlebody, in general for an arbitrary irreducible $3$-manifold with compressible boundary.
The disk complex is not only a powerful tool when we study the mapping class group of a manifold, but already an interesting topological object by itself.
For a Heegaard splitting of a manifold, one can define a special subcomplex of the disk complex for one of the handlebodies in the splitting, which is invariant under any automorphism of the manifold preserving the given splitting.
In particular, when a given splitting is stabilized, we define a natural subcomplex of the disk complex for one of the handlebodies, which is called the primitive disk complex.
The primitive disk complex is the full subcomplex spanned by special kind of vertices, the vertices of primitive disks in the handlebody.

The group of automorphisms preserving a given splitting of genus $g$ is called the genus $g$ Goeritz group of the splitting.
Investigating the action of the Goeritz group on an appropriate subcomplex, one might have a useful information on this group.
For example, in \cite{C} and \cite{C2}, the primitive disk complex for the genus two Heegaard splitting for each of the $3$-sphere and the lens spaces $L(p, 1)$ is studied to obtain a finite presentation of the genus two Goeritz group of the splitting.
The study of the primitive disk complexes was motivated by Scharlemann's work \cite{Sc} on the complex of reducing spheres for the genus two splitting of the $3$-sphere, which was analyzed further by Akbas \cite{Ak}.
A generalized version of a primitive disk complex is also studied in \cite{Kod} for the group of automorphisms of the $3$-sphere preserving a genus two handlebody knot embedded in the $3$-sphere.

In this work, we are interested in the primitive disk complex for the genus two Heegaard splitting of each lens space $L(p, q)$.
The main results are Theorems \ref{thm:contractible} and \ref{thm:structure}.
In Theorem \ref{thm:contractible}, we find all lens spaces whose primitive disk complexes for their genus two splittings are contractible. It states that, for a lens space $L(p, q)$ with $1 \leq q \leq p/2$, the primitive disk complex for the genus two splitting of $L(p, q)$ is contractible if and only if $p \equiv \pm 1 \pmod{q}$.
In Theorem \ref{thm:structure}, we give a complete combinatorial structure of the primitive disk complex for each lens space, and describe how each of the complexes sits in the ambient disk complex for the genus two handlebody.
In particular, a contractible primitive disk complex is two dimensional if and only if $q=2$ or $p=2q+1$, and otherwise it is a tree.
Any non-contractible primitive disk complex turns out to be not connected and consists of infinitely but countably many tree components.

In the follow-up work \cite{CKod}, the results on the primitive disk complexes we obtained here will be fully used to find a presentation of the genus two Goeritz group of each lens space $L(p, q)$, including the special case of $L(p, 1)$ done in \cite{C2}.

We use the standard notation $L = L(p, q)$ for a lens space in standard textbooks.
For example, we refer \cite{Ro} to the reader.
That is, there is a genus one Heegaard splitting of $L$ such that an oriented meridian circle of a solid torus of the splitting is identified with a $(p, q)$-curve on the boundary torus of the other solid torus (fixing oriented longitude and meridian circles of the torus), where $\pi_1(L(p, q))$ is isomorphic to the cyclic group of order $|p|$.
The integer $p$ can be assumed to be positive, and it is well known that two lens spaces $L(p, q)$ and $L(p', q')$ are homeomorphic if and only if $p = p'$ and  $q'q^{\pm 1} \equiv \pm 1 \pmod p$.
Thus we will assume that $0 < q < p$ for a given $L(p, q)$, or even that $1 \leq q \leq p/2$ sometimes.
Further, there is a unique integer $q'$ satisfying $1 \leq q' \leq p/2$ and $qq' \equiv \pm 1 \pmod p$, and so, for any other genus one Heegaard splitting of $L(p, q)$, we may assume that an oriented meridian circle of a solid torus of the splitting is identified with a $(p, \bar q)$-curve on the boundary torus of the other solid torus for some $\bar q \in \{q, q', p-q', p-q\}$.

\smallskip

Throughout the paper, $(V, W; \Sigma)$ will denote a genus two Heegaard splitting of a lens space $L = L(p, q)$.
That is, $V$ and $W$ are genus two handlebodies such that $V \cup W = L$ and $V \cap W = \partial V = \partial W = \Sigma$, a genus two closed orientable surface in $L$.
Any disks in a handlebody are always assumed to be properly embedded, and their intersection is transverse and minimal up to isotopy.
In particular, if a disk $D$ intersects a disk $E$, then $D \cap E$ is a collection of pairwise disjoint arcs that are properly embedded in both $D$ and $E$.
Finally, $\Nbd(X)$ will denote a regular neighborhood of $X$ and $\cl(X)$ the closure of $X$ for a subspace $X$ of a polyhedral space, where the ambient space will always be clear from the context.

\section{Primitive elements of the free group of rank two}
\label{sec:free_group}

The fundamental group of the genus two handlebody is the free group $\mathbb Z \ast \mathbb Z$ of rank two.
We call an element of $\mathbb Z \ast \mathbb Z$ {\it primitive} if it is a member of a generating pair of $\mathbb Z \ast \mathbb Z$.
Primitive elements of $\mathbb Z \ast \mathbb Z$ have been well understood.
For example, given a generating pair $\{y, z\}$ of $\mathbb Z \ast \mathbb Z$, a cyclically reduced form of any primitive element $w$ can be written as a product of terms each of the form $y^\epsilon z^n$ or $y^\epsilon z^{n+1}$, or else a product of terms each of the form $z^\epsilon y^n$ or $z^\epsilon y^{n+1}$, for some $\epsilon \in \{1,-1\}$ and some $n \in \mathbb Z$.
Consequently, no cyclically reduced form of $w$ in terms of $y$ and $z$ can contain $y$ and $y^{-1}$ $($and $z$ and $z^{-1})$ simultaneously.
Furthermore, we have the explicit characterization of primitive elements containing only positive powers of $y$ and $z$ as follows, which is given in \cite{OZ}.

\begin{lemma}
Suppose that $w$ consists of exactly $m$ $z$'s and $n$ $y$'s where $1 \leq m \leq n$.
Then $w$ is primitive if and only if $(m, n) = 1$ and $w$ has the following cyclically reduced form
$$w = w(m, n) = g(1)g(1+m)g(1+2m)\cdots g(1+(m+n-1)m)$$
where the function $g:\mathbb Z \rightarrow \{z, y\}$ is defined by
$$g(i)= g_{m, n}(i) =
\begin{cases}
~z &  \ \text{if~ ~$i \equiv 1, 2, \cdots, m  \pmod{(m+n)}$} \\
~y &  \ \text{otherwise.}
\end{cases}
$$
\label{lem:primitive}
\end{lemma}

For example, $w(3, 5) = zy^2zy^2zy$ and
$w(3, 10) = zy^4zy^3zy^3$.
One can write the word $w(m, n)$ quickly using consecutive $m$ black points for $z$'s and $n$ white points for $y$'s on a circle.
Starting at the first black one, move to the $m$-th point in the clockwise direction and so on.
And then after $m+n-1$ moves, we have $w(m, n)$ (see Figure \ref{primitive} (a)).

\begin{figure}
\centering
\includegraphics{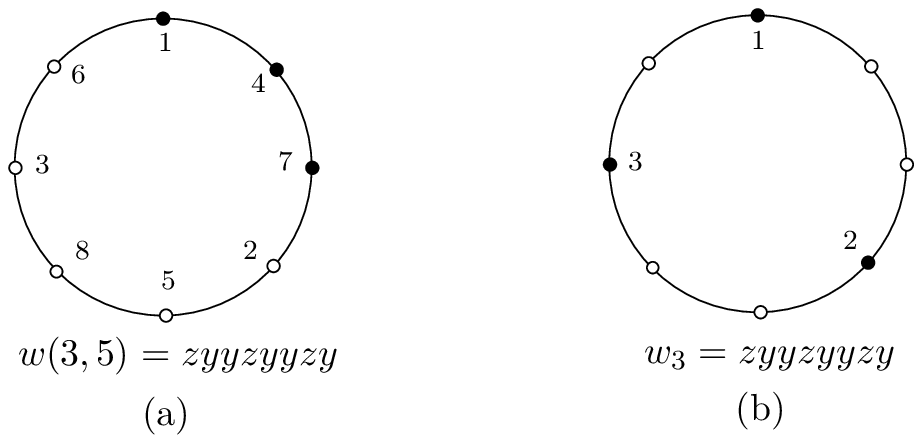}
\caption{}
\label{primitive}
\end{figure}

Let $\{z, y\}$ be a generating pair of the free group of rank two.
Given relatively prime integers $p$ and $q$ with $0 < q <p$, we define a sequence of $(p+1)$ elements $w_0, w_1, \cdots, w_{p-1}, w_p$ in term of $z$ and $w$ as follows.

Define first $w_0$ to be $y^p$.
For each $j \in \{1, 2, \cdots, p\}$, let $f_j:\mathbb Z \rightarrow \{z, y\}$ be the function given by
$$f_j(i)=
\begin{cases}
~z &  \ \text{if ~$i \equiv 1, 1+q, 1+2q, \cdots, 1+(j-1)q \pmod{p}$} \\
~y &  \ \text{otherwise,}
\end{cases}
$$
and then define $w_j = f_j(1)f_j(2)\cdots f_j(p)$.
Each of $w_j$ has length $p$ and consists of $j$ $z$'s and $(p-j)$ $y$'s.
In particular, $w_1 = zy^{p-1}$, $w_{p-1} = z^{p-q}yz^{q-1}$ and $w_p = z^p$.
We call the sequence $w_0, w_1, \cdots w_p$ the {\it $(p, q)$-sequence} of the pair $(z, y)$.
An easy way to write down the word $w_j$ is to put $p$ white points on a circle, and then change one to a black point, and change the $q$-th white one to black in clockwise direction and so one. Each black is $z$ and white $y$, and after $j$ changes, we have $w_j$ (see Figure \ref{primitive} (b), $w_3$ in $(8, 3)$-sequence).
For example, the $(8, 3)$-sequence is given by
\begin{alignat*}{3}
w_0 &= yyyyyyyy &\qquad
w_1 &= zyyyyyyy &\qquad
w_2 &= zyyzyyyy \\
w_3 &= zyyzyyzy &\qquad
w_4 &= zzyzyyzy &\qquad
w_5 &= zzyzzyzy \\
w_6 &= zzyzzyzz &\qquad
w_7 &= zzzzzyzz &\qquad
w_8 &= zzzzzzzz
\end{alignat*}
Observe that $w_{p-j}$ is a cyclic permutation of $\overline{\psi(w_j)}$ for each $j$, where $\psi$ is the automorphism exchanging $z$ and $y$, and $\overline{w}$ is the reverse of $w$.
Thus $w_j$ is primitive if and only if $w_{p-j}$ is primitive.
We can find all primitive elements in the sequence as follows.

\begin{lemma}[Four Primitives Lemma]
Let $w_0, w_1, \cdots, w_p$ be the $(p, q)$-sequence of the generating pair $\{z, y\}$ with $0< q < p$.
Let $q'$ be the unique integer satisfying $1 \leq q' \leq p/2$ with $qq' \equiv \pm 1 \pmod{p}$.
Then $w_j$ is primitive if and only if $j \in \{1, q', p-q', p-1\}$.
\label{lem:four_primitives}
\end{lemma}

\begin{proof}
It is clear that $w_1$ and $w_{p-1}$ are primitive while $w_0$ and $w_p$ are not.

\medskip

{\noindent \sc Claim 1.} $w_{q'}$ is primitive.\\
{\it Proof of Claim 1.}
We write $w_{q'} = f_{q'}(1)f_{q'}(2) \cdots f_{q'}(p)$, and $w(q', p-q') = g(1)g(1+q')g(1+2q')\cdots g(1+(p-1)q')$ where $g = g_{q', p-q'}$ in the notation in Lemma \ref{lem:primitive}.
Since $f(i) = z$ if and only if $i \equiv 1+nq \pmod{p}$ for some $n \in \{0, 1, \cdots, q'-1\}$, it can be directly verified that
$$f_{q'}(i)=
\begin{cases}
 g(1+(i-1)q')&  \ \text{if~ $qq' \equiv 1 \pmod{p}$} \\
 g(1+(i+q)q')&  \ \text{if~ $qq' \equiv -1 \pmod{p}$}.
\end{cases}
$$
Thus $w_{q'}$ is $w(q', p-q')$ itself if $qq' \equiv 1 \pmod{p}$ or is a cyclic permutation of it if $qq' \equiv -1 \pmod{p}$.
In either cases, $w_{q'}$ is primitive.

\medskip

{\noindent \sc Claim 2.} If $1 < j \leq p/2$ and $j \neq q'$, then $w_j$ is not primitive.\\
{\it Proof of Claim 2.}
From the assumption, there is a unique integer $r$ satisfying $2 \leq r \leq p-2$ and $qj \equiv r \pmod{p}$.
Suppose, for contradiction, that $w_j$ is primitive.
Then, by Lemma \ref{lem:primitive}, $(p, j) = 1$ and $w_j$ is a cyclic permutation of $w(j, p-j)$.
We write $w_j = f_j(1) f_j(2) \cdots  f_j(p)$ and $w(j, p-j) = g(1)g(1+j)g(1+2j) \cdots g(1+(p-1)j)$ where $g = g_{j, p-j}$ as in Lemma \ref{lem:primitive}.
Then there is a constant $k$ such that $f_j(i) = g(1+ (i-1+k)j)$ for all $i \in \mathbb Z$.
In particular, $f_j(1 + nq) = z = g(1+(nq + k)j)$ for each $ n \in \{0, 1, \cdots, j-1\}$.

From the definition of $g = g_{j, p-j}$ and the choice of the integer $r$, we have $1+(nq + k)j \equiv 1+nr+kj \equiv 1, 2, \cdots, j \pmod{p}$.
Let $a_n$ be the unique integer satisfying $1+nr + kj \equiv a_n$ with $a_n \in \{1, 2, \cdots, j\}$ for each $n \in \{0, 1, \cdots, j-1\}$.
Observe that $a_n + r \equiv a_{n+1}$ for each $n \in \{0, 1, \cdots, j-2 \}$, and in particular, $a_0 + r \equiv a_1$.
Since $1 \leq a_0 \leq j < p$ and $2 \leq r \leq p-2 < p$, we have only two possibilities: either $a_0 + r = a_1$ or $a_0 + r = a_1 + p$.

First consider the case  $a_0 + r = a_1$.
Then $r \leq j-1$ and $a_n < a_{n+1}$, consequently $a_0 =1, a_1 = 2, \cdots, a_{j-1} = j$, which implies $r=1$, a contradiction.
Next, if $a_0 + r = a_1 + p$, then $p+1-j \leq r$ and $a_n > a_{n+1}$, thus we have $a_0 = j, a_1 = j-1, \cdots, a_{j-1} = 1$, and consequently $r = p-1$, a contradiction again.

\medskip

By the claims, if $1 \leq j \leq p/2$, then $w_j$ is primitive only when $j =1$ or $j=q'$.
If $p/2 \leq j \leq p$, due to the fact that $w_{p-j}$ is a cyclic permutation of $\overline{\psi(w_j)}$, the only primitive elements are $w_{p-q'}$ and $w_{p-1}$, which completes the proof.
\end{proof}

A simple closed curve in the boundary of a genus two handlebody $W$ represents elements of $\pi_1(W) = \mathbb Z \ast \mathbb Z$.
We call a pair of essential disks in $W$ a {\it complete meridian system} for $W$ if the union of the two disks cuts off $W$ into a $3$-ball.
Given a complete meridian system $\{D, E\}$, assign symbols $x$ and $y$ to the circles $\partial D$ and $\partial E$ respectively.
Suppose that an oriented simple closed curve $l$ on $\partial W$ that meets $\partial D \cup \partial E$ transversely and minimally.
Then $l$ determines a word in terms of $x$ and $y$ which can be read off from the the intersections of $l$ with $\partial D$ and $\partial E$ (after a choice of orientations of $\partial D$ and $\partial E$), and hence $l$ represents an element of the free group $\pi_1 (W) = \left< x, y\right>$.

In this set up, the following is a simple criterion for the primitiveness of the elements represented by such simple closed curves.

\begin{lemma}
With a suitable choice of orientations of $\partial D$ and $\partial E$, if a word corresponding to a simple closed curve $l$ contains one of the pairs of terms$:$ $(1)$  both of $xy$ and $xy^{-1}$ or $(2)$  both of $xy^nx$ and $y^{n+2}$ for $n \geq 0$, then the element of $\pi_1 (W)$ represented by $l$ cannot be $($a positive power of $)$ a primitive element.
\label{lem:key}
\end{lemma}

\begin{figure}
\centering
\includegraphics{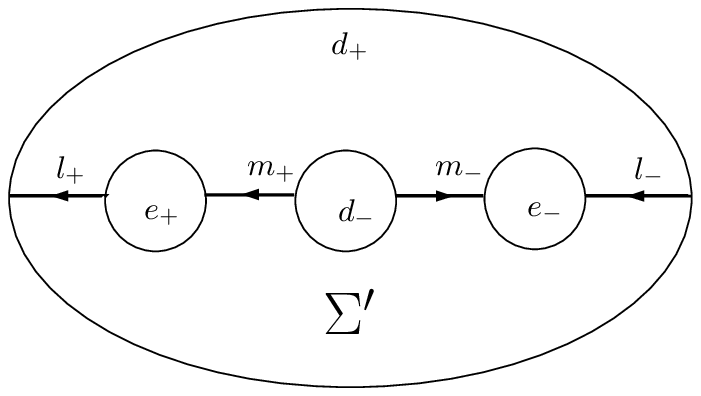}
\caption{}
\label{not_primitive}
\end{figure}

\begin{proof}
Let $\Sigma'$ be the $4$-holed sphere cut off from $\partial W$ along $\partial D \cup \partial E$.
Denote by $d_+$ and $d_-$ (by $e_+$ and $e_-$ respectively) the boundary circles of $\Sigma'$ that came from $\partial D$ (from $\partial E$ respectively).

Suppose first that $l$ represents an element of a form containing both $xy$ and $xy^{-1}$.
Then we may assume that there are two subarcs $l_+$ and $l_-$ of $l \cap \Sigma'$ such that $l_+$ connects $d_+$ and $e_+$, and $l_-$ connects $d_+$ and $e_-$ as in Figure \ref{not_primitive}.
Since $|~l \cap d_+| = |~l \cap d_-|$ and $|~l \cap e_+| = |~l \cap e_-|$, we must have two other arcs $m_+$ and $m_-$ of $l \cap \Sigma'$ such that $m_+$ connects $d_-$ and $e_+$, and $m_-$ connects $d_-$ and $e_-$.

Consequently, there exists no arc component of $l \cap \Sigma'$ that meets only one of $d_+$, $d_-$, $e_+$ and $e_-$.
That is, any word corresponding to $l$ contains neither $x^{\pm1} x^{\mp1}$ nor $y^{\pm1} y^{\mp1}$, and hence it is cyclically reduced.
Considering all possible directions of the arcs $l_+$, $l_-$, $m_+$ and $m_-$, each word represented by $l$ must contain both $x$ and $x^{-1}$ (or both $y$ and $y^{ -1}$), which means that $l$  cannot represent (a positive power of) a primitive element of $\pi_1 (W)$.

Next, suppose that a word corresponding to $l$ contains $x^2$ and $y^2$, which is the case of $n=0$ in the second condition.
Then there are two arcs $l_+$ and $l_-$ of $l \cap \Sigma'$ such that $l_+$ connects $d_+$ and $d_-$, and $l_-$ connects $e_+$ and $e_-$.
By a similar argument to the above, we see again that any word corresponding to $l$ is cyclically reduced, but contains both of $x^2$ and $y^2$.
Thus $l$ cannot represent (a positive power of) a primitive element.

Suppose that a word corresponding to $l$ contains $xy^nx$ and $y^{n+2}$ for $n \geq 1$.
Then there are two subarcs $\alpha$ and $\beta$ of $l$ which correspond to $xy^nx$ and $y^{n+2}$ respectively.
In particular, we may assume that $\alpha$ starts at $d_+$, intersects $\partial E$ in $n$ points, and ends in $d_-$, while $\beta$ starts at $e_+$, intersects $\partial E$ in its interior in $n$ points, and ends in $e_-$.

Let $m$ be the subarc of $\alpha$ corresponding to $xy$.
Then $m$ connects two circles $d_+$ and one of $e_{\pm}$, say $e_+$.
Choose a disk $E^\ast$ properly embedded in the $3$-ball $W$ cut off by $D \cup E$ such that the boundary circle $\partial E^\ast$ is the frontier of a regular neighborhood of $d_+ \cup m \cup e_+$ in $\Sigma'$.
Then $E^\ast$ is a non-separating disk in $W$ and forms a complete meridian system with $D$.
Assigning the same symbol $y$ to $\partial E^\ast$, the arc $\alpha$ determines $xy^{n-1}x$ while $\beta$ determines $y^{n+1}$.
Thus the conclusion follows by induction.
\end{proof}

\section{Primitive disks in a handlebody}
\label{sec:primitive_disks}

Recall that $(V, W; \Sigma)$ denotes a genus two Heegaard splitting of a lens space $L = L(p, q)$ with $0 < q < p$.
An essential disk $E$ in $V$ is called {\it primitive} if there exists an essential disk $E'$ in $W$ such that $\partial E$ intersects $\partial E' $ transversely in a single point.
Such a disk $E'$ is called a {\it dual disk} of $E$, which is also primitive in $W$ having a dual disk $E$.
Note that both $W \cup \Nbd(E)$ and $V \cup \Nbd(E')$ are solid tori.
Primitive disks are necessarily non-separating.
We call a pair of disjoint, non-isotopic primitive disks in $V$ a {\it primitive pair} in $V$.
Similarly, a triple of pairwise disjoint, non-isotopic primitive disks is a {\it primitive triple}.

A non-separating disk $E_0$ properly embedded in $V$ is called {\it semiprimitive} if there is a primitive disk $E'$ in $W$ disjoint from $E_0$.
The boundary circle $\partial E_0$ can be considered as a $(p, \bar q)$-curve on the boundary of the solid torus $W$ cut off by $E'$ for some integer $\bar q$ satisfying the relation $\bar q q^{\pm 1} \equiv \pm 1 \pmod{p}$.

Any simple closed curve on the boundary of the solid torus $W$ represents an element of $\pi_1 (W)$ which is the free group of rank two.
We interpret primitive disks algebraically as follows, which is a direct consequence of \cite{Go}.

\begin{lemma}
Let $D$ be a non-separating disk in $V$.
Then $D$ is primitive if and only if $\partial D$ represents a primitive element of $\pi_1 (W)$.
\label{lem:primitive_element}
\end{lemma}

Note that no disk can be both primitive and semiprimitive since the boundary circle of a semiprimitive disk in $V$ represents the $p$-th power of a primitive element of $\pi_1(W)$.

Let $D$ and $E$ be essential disks in $V$, and suppose that $D$ intersects $E$ transversely and minimally.
Let $C \subset D$ be a disk cut off from $D$ by an outermost arc $\alpha$ of $D \cap E$ in $D$ such that $C \cap E= \alpha$.
We call such a $C$ an {\it outermost subdisk} of $D$ cut off by $D \cap E$.
The arc $\alpha$ cuts $E$ into two disks, say $G$ and $H$.
Then we have two disjoint disks $E_1$ and $E_2$ which are isotopic to disks $G \cup C$ and $H \cup C$ respectively.
We call $E_1$ and $E_2$ the {\it disks from surgery} on $E$ along the outermost subdisk $C$ of $D$.

Since $E$ and $D$ are assumed to intersect minimally, $E_1$ (and $E_2$) is isotopic to neither $E$ nor $D$.
Also both of $E_1$ and $E_2$ are non-separating if $D$ is non-separating.
Observe that each of $E_1$ and $E_2$ has fewer arcs of intersection with $D$ than $E$ had since at least the arc $\alpha$ no longer counts.

For an essential disk $D$ in $V$ intersecting  transversely and minimally the union of two disjoint essential disks $E$ and $F$, we define similarly the disks from surgery on $E \cup F$ along an outermost subdisk of $D$ cut off by $D \cap (E \cup F)$.
In this case, if $D$ is non-separating and $\{E, F\}$ is a complete meridian system of $V$, then one of the two disks is isotopic to $E$ or $F$ and the other one to neither of $E$ and $F$.

\begin{lemma}
Let $\{D, E\}$ be a primitive pair of $V$.
Then $D$ and $E$ have a common dual disk if and only if there is a semiprimitive disk $E_0$ in $V$ disjoint from $D$ and $E$.
\label{lem:common_dual}
\end{lemma}

\begin{proof}
The necessity is clear.
For sufficiency, let $E'$ be a primitive disk in $W$ disjoint from the semiprimitive disk $E_0$ in $V$.
It is enough to show that $E'$ is a dual disk of every primitive disk in $V$ disjoint from $E_0$, since then $E'$ would be a common dual disk of $D$ and $E$.

\smallskip

\noindent {\it Claim}: If $E$ is a primitive disk in $V$ dual to $E'$, then $E$ is disjoint from $E_0$.

\noindent {\it Proof of claim}.
Denote by $E_0^+$ and $E_0^-$ the two disks on the boundary of the solid torus $V$ cut off by $E_0$ that came from $E_0$.
Suppose that $E$ intersects $E_0$.
Then an outermost subdisk $C$ of $E$ cut off by $E \cap E_0$ must intersect $\partial E'$ since $\partial E'$ is a longitude of the solid torus $V$ cut off by $E_0$.
We may assume that $C$ is incident to $E_0^+$.
Considering $|E \cap E_0^+| = |E \cap E_0^-|$, there is a subarc of $\partial E$ whose two endpoints lie in $\partial E_0^-$, which also intersects $\partial E'$, and hence $\partial E$ intersects $\partial E'$ at least in two points, a contradiction.

\smallskip

Let $D$ be a primitive disk in $V$ disjoint from $E_0$.
Among all the primitive disks in $V$ dual to $E'$, choose one, denoted by $E$ again, such that $|D \cap E|$ is minimal.
By the claim, $E$ is disjoint from $E_0$.
Let $E'_0$ be the unique semiprimitive disk in $W$ disjoint from $E \cup E'$.
Since $\{E', E'_0\}$ forms a complete meridian system of $W$, by assigning symbols $x$ and $y$ to oriented $\partial E'$ and $\partial E'_0$ respectively, any oriented simple closed curve on $\partial W$ represents an element of the free group $\pi_1 (W) = \langle x, y \rangle$ as in the previous section.
In particular, we may assume that $\partial E$ and $\partial E_0$ represents elements of the form $x$ and $y^p$ respectively.

Denote by $\Sigma_0$ the $4$-holed sphere $\partial V$ cut off by $\partial E \cup \partial E_0$.
Consider $\Sigma_0$ as a $2$-holed annulus with two boundary circles $\partial E_0^\pm$ came from $\partial E_0$ and with two holes $\partial E^\pm$ came from $\partial E$.
Then $\partial E'_0$ is the union of $p$ spanning arcs in $\Sigma_0$ which divides $\Sigma_0$ into $p$ rectangles, and the two holes $\partial E^\pm$ is contained in a single rectangle.
Notice that $\partial E'$ is an arc in the rectangle connecting the two holes.
See Figure \ref{Sigma_0} (a).

\begin{figure}
\centering
\includegraphics[width=10cm,clip]{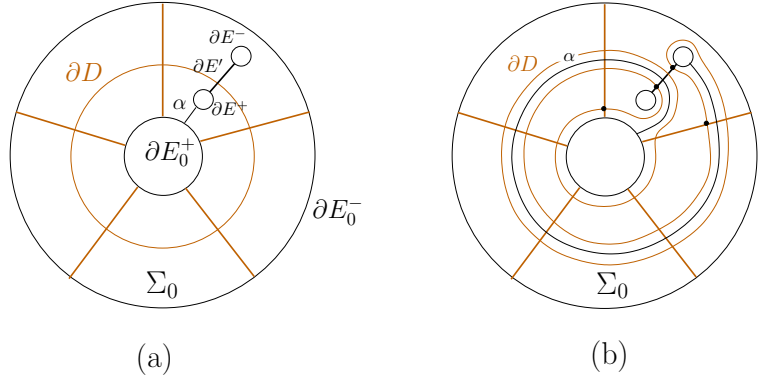}
\caption{The $2$-holed annulus $\Sigma_0$ when $p = 5$, for example.}
\label{Sigma_0}
\end{figure}

Suppose that $D$ is disjoint from $E$.
Then $D$ is a non-separating disk in $V$ disjoint from $E \cup E_0$, and hence the boundary circle $\partial D$ can be considered as the frontier of a regular neighborhood in $\Sigma_0$ of the union of one of the two boundary circles, one of the two holes of $\Sigma_0$, and an arc $\alpha$ connecting them.
The arc $\alpha$ cannot intersect $\partial E'_0$ in $\Sigma_0$, otherwise an element represented by $\partial D$ must contain both of $xy$ and $xy^{-1}$ (after changing orientations if necessary), which contradicts that $D$ is primitive by Lemma \ref{lem:key} (see Figure \ref{Sigma_0} (b)).
Thus $\alpha$ is disjoint from $\partial E'_0$, and consequently $D$ intersects $\partial E'$ in a single point.
That is, $E'$ is a dual disk of $D$ (see Figure \ref{Sigma_0} (a)).

Suppose next that $D$ intersects $E$.
Let $C$ be an outermost subdisk of $D$ cut off by $D \cap E$.
Then one of the resulting disks from surgery on $E$ along $C$ is $E_0$ and the other one, say $E'$, is isotopic to none of $E$ and $E_0$.
The arc $\partial C \cap \Sigma_0$ can be considered as the frontier of a regular neighborhood of the union of a boundary circle of $\Sigma_0$ came from $\partial E_0$ and an arc, denoted by $\alpha_0$, connecting this circle and a hole came from $\partial E$.
By a similar argument to the above, one can show that $\alpha_0$ is disjoint from $\partial E'_0$, otherwise $D$ would not be primitive.
Consequently, the boundary circle of the resulting disk $E_1$ from the surgery intersects $\partial E'$ in a single point, which means $E_1$ is primitive with the dual disk $E'$.
But we have $|D \cap E_1| < |D \cap E|$ from the surgery construction, which contradicts the minimality of $|D \cap E|$.
\end{proof}

In the proof of Lemma \ref{lem:common_dual}, if we assume that the primitive disk $D$ also intersects $E_0$, then the subdisk $C$ of $D$ cut off by $D \cap (E \cup E_0)$ would be incident to one of $E$ and $E_0$. The argument to show that the resulting disk $E_1$ from the surgery is primitive with the dual disk $E'$ still holds when $C$ is incident to $E_0$ and even when $D$ is semiprimitive. This observation suggests the following lemma.

\begin{lemma}
Let $E_0$ be a semiprimitive disk in $V$ and let $E$ be a primitive disk in $V$ disjoint from $E_0$.
If a primitive or semiprimitive disk $D$ in $V$  intersects $E \cup E_0$, then one of the disks from surgery on $E \cup E_0$ along an outermost subdisk of $D$ cut off by $D \cap (E \cup E_0)$ is either $E$ or $E_0$, and the other one, say $E_1$, is a primitive disk, which has a common dual disk with $E$.
\label{lem:first_surgery}
\end{lemma}

\section{A sequence of disks generated by a dual pair}
\label{sec:sequence_of_disks}
In this section, we introduce a special sequence of disks in the handlebody $V$, which will play a key role in the following sections.
Again, we have a genus two Heegaard splitting $(V, W; \Sigma)$ of a lens space $L = L(p, q)$ with $0 < q < p$.
Let $E$ be a primitive disk in $V$ with a dual disk $E'$, and let $E_0$ and $E'_0$ be the unique semiprimitive disks in $V$ and $W$ respectively which are disjoint from $E \cup E'$.
We will assume that $\partial E'_0$ is a $(p, q)$-curve on the boundary of the solid torus $\cl(V- \Nbd(E))$ with $0 < q < p$.

Assigning symbols $x$ and $y$ to oriented $\partial E'$ and $\partial E'_0$ respectively as in the previous sections, any oriented simple closed curve on $\partial W$ represents an element of the free group $\pi_1 (W) = \langle x, y \rangle$.
We simply denote the circles $\partial E'$ and $\partial E'_0$ by $x$ and $y$ respectively.
The circle $y$ is disjoint from $\partial E$ and intersects $\partial E_0$ in $p$ points, and $x$ is disjoint from $\partial E_0$ and intersects $\partial E$ in a single point.
Thus we may assume that $\partial E_0$ and $\partial E$ determine the elements of the form $y^p$ and $x$ respectively.

Let $\Sigma_0$ be the $4$-holed sphere $\partial V$ cut off by $\partial E \cup \partial E_0$.
Denote by $e^\pm$ the boundary circles of $\Sigma_0$ came from $\partial E$ and similarly $e_0^\pm$ came from $\partial E_0$.
The $4$-holed sphere $\Sigma_0$ can be regarded as a $2$-holed annulus where the two boundary circles are $e^\pm_0$ and the two holes $e^\pm$.
Then the circle $y$ in $\Sigma_0$ is the union of $p$ spanning arcs which cuts the annulus into $p$ rectangles, and $x$ is a single arc connecting two holes $e^\pm$, where $x \cup e^\pm$ is contained in a single rectangle (see the surface $\Sigma_0$ in Figure \ref{sequence}).

\begin{figure}
\centering
\includegraphics{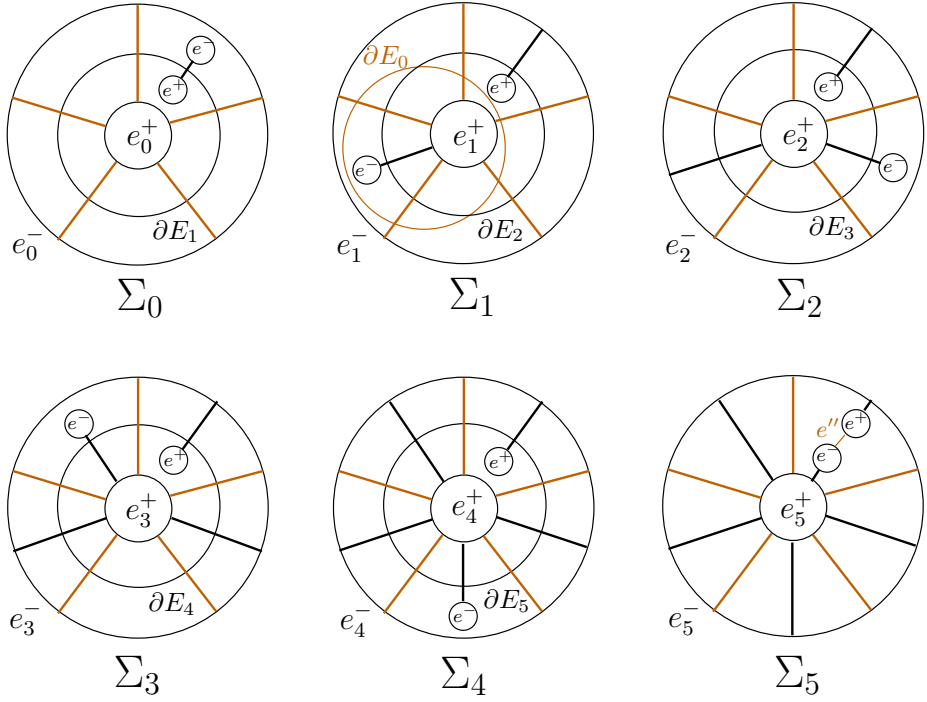}
\caption{A sequence of disks generated by a dual pair for $L(5, 3)$ where $\partial E_0'$ is a $(5, 3)$-curve.}
\label{sequence}
\end{figure}

Any non-separating disk in $V$ disjoint from $E \cup E_0$ and not isotopic to either of $E$ and $E_0$ is determined by an arc properly embedded in $\Sigma_0$ connecting one of $e^\pm$ and one of $e^\pm_0$.
That is, the boundary circle of such a disk is the frontier of a regular neighborhood of the union of the arc and the two circles connected by the arc in $\Sigma_0$.
Choose such an arc $\alpha_0$ so that $\alpha_0$ is disjoint from $y$, and denote by $E_1$ the non-separating disk determined by $\alpha_0$.
Observe that there are infinitely many choices of such arcs up to isotopy, and so are the disks $E_1$, but the element represented by $\partial E_1$ has one of the forms $x^{\pm 1} y^{\pm p}$, so we may assume that $\partial E_1$ represents $xy^p$ by changing the orientations if necessary.

Next, let $\Sigma_1$ be the $4$-holed sphere $\partial V$ cut off by $\partial E \cup \partial E_1$.
As in the case of $\Sigma_0$, consider $\Sigma_1$ as a $2$-holed annulus with boundaries $e^\pm_1$ and with two holes $e^\pm$ where $e^\pm_1$ came from $\partial E_1$.
Then the circle $y$ cuts off $\Sigma_1$ into $p$ rectangles as in the case of $\Sigma_0$, but two holes $e^+$ and $e^-$ are now contained in different rectangles.
In particular, we can give labels $0,1,\ldots, p-1$ to the rectangles consecutively so that $e^+$ lies in the rectangle labeled by $0$ while $e^-$ in that by $q$.
The circle $x$ in $\Sigma_1$ is the union of two arcs connecting $e^\pm_1$ and $e^\pm$ contained in the rectangles labeled by $0$ and $p$ respectively.

Now consider a properly embedded arc in $\Sigma_1$ connecting one of $e^\pm$ and one of $e_1^\pm$.
Choose such an arc $\alpha_1$ so that $\alpha_1$ is disjoint from $y$ and parallel to none of the two arcs of $x \cap \Sigma_1$.
Then $\alpha_1$ determines a non-separating disk, denoted by $E_2$, whose boundary circle is the frontier of a regular neighborhood of the union of $\alpha_1$ and the two circles connected by $\alpha_1$.
(If $\alpha_1$ is isotopic to one of the two arcs $x \cap \Sigma_1$, then the resulting disk is $E_0$.)
Observe that $\partial E_2$ represents an element of the form $xy^q xy^{p-q}$ (see the surface $\Sigma_1$ in Figure \ref{sequence}).

We continue this process in the same way.
Then $\Sigma_2$ is the $4$-holed sphere $\partial V$ cut off by $\partial E \cup \partial E_2$, and we choose an arc $\alpha_2$ in $\Sigma_2$ disjoint from $y$ and parallel to none of the arcs $x \cap \Sigma_2$, which determines the disk $E_3$.
The boundary circle $\partial E_3$ represents an element of the form $xy^qxy^qxy^{p-2q}$ if $0 < q \leq p/2$ or $xy^{2q-p}xy^{p-q}xy^{p-q}$ if $p/2 <q <p$.
In general, we have a non-separating disk $E_j$ whose boundary circle lies in the $4$-holed sphere $\Sigma_{j-1}$.

We finish the process in the $p$-th step to have the disk $E_p$ whose boundary circle lies in $\Sigma_{p-1}$.
The disk $E_{p-1}$ and $E_p$ represent elements of the form $(xy)^{p-q}y(xy)^{q-1}$ and $(xy)^p$ respectively.
Observe that there are infinitely many choices of the arc $\alpha_0$, and so choices of the disk $E_1$ as we have seen, but once $E_1$ have been chosen, the next disks $E_j$ for each $j \in \{1, 2, \cdots, p-1\}$ are uniquely determined.
We call the sequence of the disks $E_0, E_1, \cdots, E_{p-1}, E_p$ in $V$ a {\it sequence of disks} generated by the dual pair $\{E, E'\}$.

Note that $E_j$ is disjoint from $E_{j+1}$, and intersects $E_{j+2}$ in a single arc for each $j \in \{0, 1, \cdots, p-2\}$.
For example, see $\partial E_0$, $\partial E_2$ and $\partial E_1$ ($= e_1^{\pm}$) in $\Sigma_1$ in Figure \ref{sequence}.
In general, we have $|E_i \cap E_j| = j - i -1$ for $0 \leq i < j \leq p$.
This is obvious from the construction.
Figure \ref{intersection} illustrates intersections of $E_j$ with $E_{j+2}$, $E_{j+3}$ and $E_{j+4}$ in the $3$-balls $V$ cut off by $E \cup E_{j+1}$, $E \cup E_{j+2}$ and $E \cup E_{j+3}$ respectively.

\begin{figure}
\centering
\includegraphics[width=12.5cm,clip]{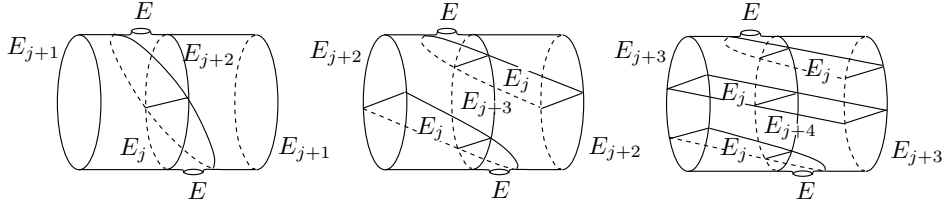}
\caption{Intersections of $E_j$ with $E_{j+2}$, $E_{j+3}$ and $E_{j+4}$}
\label{intersection}
\end{figure}

\begin{lemma}
Let $E_0, E_1, \cdots, E_{p-1}, E_p$ be a sequence of disks in $V$ generated by a dual pair $\{E, E'\}$.
Then we have
\begin{enumerate}
\item $E_0$ and $E_p$ are semiprimitive.
\item $E_j$ is primitive if and only if $j \in \{1, q', p-q', p-1\}$ where $q'$ is the unique integer satisfying $q'q \equiv \pm1 \pmod p$ and $1 \leq q' \leq p/2$.
\end{enumerate}
\label{lem:sequence}
\end{lemma}

\begin{proof}
(1) $E_0$ is a semiprimitive disk disjoint from $E'$ from the construction.
For the disk $E_p$, it is easy to find a circle $e''$ in $\Sigma$ such that $e'' \cap \Sigma_p$ is an arc which connects the two holes $e^+$ and $e^-$ and is disjoint from $x \cup y \cup e^+_p \cup e^-_p$ (see the arc $e''$ in the surface $\Sigma_5$ in Figure \ref{sequence}).
Cutting $W$ along $E' \cup E'_0$, we have a $3$-ball $B$, and the circle $e''$ lies in $\partial B$.
Thus $e''$ bounds a disk $E''$ in $W$ which is primitive since $e''$ intersects $\partial E$ in a single point.
The disk $E_p$ is disjoint from $E''$ and so is semiprimitive.

(2) From the construction, each circle $\partial E_j$ represents the element $w_j$ in the $(p, q)$-sequence in section \ref{sec:free_group}, by the substitution of $z$ for $xy$.
Thus the conclusion follows by Lemma \ref{lem:four_primitives} with Lemma \ref{lem:primitive_element}.
\end{proof}

The following is a kind of generalization of Lemma \ref{lem:first_surgery} into the case of the sequence of disks generated by a dual pair.

\begin{lemma}
Let $E_0, E_1, \cdots, E_{p-1}, E_p$ be a sequence of disks in $V$ generated by a dual pair $\{E, E'\}$ where $E$ is in $V$, and let $D$ be a primitive or semiprimitive disk in $V$.
For $j \in \{1, 2, \cdots, p-1\}$,
\begin{enumerate}
\item if $D$ is disjoint from $E \cup E_j$ and is isotopic to none of $E$ and $E_j$, then $D$ is isotopic to either $E_{j-1}$ or $E_{j+1}$, and
\item if $D$ intersects $E \cup E_j$, then one of the disks from surgery on $E \cup E_j$ along an outermost subdisk $C$ of $D$ cut off by $D \cap (E \cup E_j)$ is either $E$ or $E_j$, and the other one is either $E_{j-1}$ or $E_{j+1}$.
\end{enumerate}
\label{lem:surgery_on_primitive}
\end{lemma}

\begin{proof}
Suppose that $D$ is disjoint from $E \cup E_j$.
The boundary circle $\partial D$ lies in the $2$-holed annulus $\Sigma_j$.
Thus $\partial D$ can be considered as the frontier of the union of one hole and one boundary circle of $\Sigma_j$, and an arc $\alpha_j$ connecting them.
By the same argument for the proof of Lemmas \ref{lem:common_dual} and \ref{lem:first_surgery}, the arc $\alpha_j$ cannot intersect the arcs of $\partial E'_0 \cap \Sigma_j$ otherwise $D$ would not be (semi)primitive.
Thus the disk $D$ must be either $E_{j-1}$ or $E_{j+1}$.
(Note that if both of $E_{j-1}$ and $E_{j+1}$ are not primitive, then we can say that such a primitive disk $D$ does not exist.)
The second statement can be proved in the same manner by considering the arc $\partial C \cap \Sigma_j$  for the outermost subdisk $C$ of $D$.
\end{proof}

\section{Disks obtained from surgery on a primitive disk}
\label{sec:surgery_on_primitive}

In this section, we develop several properties of primitive disks, pairs and triples.
Here we will assume that $1 \leq q \leq p/2$ for a given lens space $L(p, q)$ with a genus two Heegaard splitting $(V, W; \Sigma)$.

\begin{theorem}
Given a lens space $L(p, q)$ with $1 \leq q \leq p/2$, suppose that $p \equiv \pm 1 \pmod{q}$.
Let $D$ and $E$ be primitive disks in $V$ which intersect each other transversely and minimally.
Then at least one of the two disks from surgery on $E$ along an outermost subdisk of $D$ cut off by $D \cap E$ is primitive.\par
\label{thm:primitive}
\end{theorem}

\begin{proof}
Let $C$ be an outermost subdisk of $D$ cut off by $D \cap E$.
The choice of a dual disk $E'$ of $E$ determines a unique semiprimitive disk $E_0$ in $V$ which is disjoint from $E \cup E'$.
Among all the dual disks of $E$, choose one, denoted by $E'$ again, so that the resulting semiprimitive disk $E_0$ intersects $C$ minimally.
If $C$ is disjoint from $E_0$, then, by Lemma \ref{lem:first_surgery}, the disk from surgery on $E$ along $C$ other than $E_0$ is primitive, having the common dual disk $E'$ with $E$, and so we are done.

From now on, we assume that $C$ intersects $E_0$.
Then one of the disks from surgery on $E_0$ along an outermost subdisk $C_0$ of $C$ cut off by $C \cap E_0$ is $E$, and the other one, say $E_1$, is primitive having the common dual disk $E'$ with $E$, by Lemma \ref{lem:first_surgery} again.
Then we have the sequence $E_0, E_1, E_2, \cdots, E_p$ of disks generated by the dual pair $\{E, E'\}$ starting with the disks $E_0$ and $E_1$.
Let $E'_0$ be the unique semiprimitive disk in $W$ disjoint from $E \cup E'$.
The circle $\partial E'_0$ would be a $(p, \bar q)$-curve on the boundary of the solid torus $\cl(V- \Nbd(E \cup E'))$ for some $\bar q \in \{q, q', p-q', p-q\}$, where $q'$ satisfies $0 < q' < p$ and $qq' \equiv \pm 1 \pmod{p}$.
We will consider only the case of $\bar q = q$, and so $\partial E'_0$ is a $(p, q)$-curve again.
The proof is easily adapted for the other cases.

If $C$ intersects $E_1$, then one of the disks from surgery on $E_1$ along an outermost subdisk $C_1$ of $C$ cut off by $C \cap E_1$ is $E$, and the other one is either $E_0$ or $E_2$ by Lemma \ref{lem:surgery_on_primitive}, but it is actually $E_2$ since we have $|C \cap E_1| < |C \cap E_0|$ from the surgery construction.
In general, if $C$ intersects each of $E_1, E_2, \cdots, E_j$, for $j \in \{1, 2, \cdots, p-1\}$, the disk from surgery on $E_j$ by an outermost subdisk $C_j$ of $C$ cut off by $C \cap E_j$, other than $E$, is $E_{j+1}$, and we have $|C \cap E_{j+1}| < |C \cap E_j|$.
Consequently, we see that $|C \cap E_p| < |C \cap E_0|$, but it contradicts the minimality of $|C \cap E_0|$ since $E_p$ is also a semiprimitive disk disjoint from $E$.
Thus, there is a disk $E_j$ for some $j \in \{1, 2, \cdots, p-1\}$ which is disjoint from $C$.

Now, denote by $E_j$ again the first disk in the sequence that is disjoint from $C$.
Then the two disks from surgery on $E$ along $C$ are $E_j$ and $E_{j+1}$, hence $C$ is also disjoint from $E_{j+1}$.
Actually they are the only disks in the sequence disjoint from $C$.
For other disks in the sequence, it is easy to see that $|C \cap E_{j-k}| = k = |C \cap E_{j+1+k}|$ (by a similar observation to the fact that $|E_i \cap E_j| = j - i -1$ for $0 \leq i < j \leq p$ in the sequence of disks).
If $j \geq p/2$, then we have $|C \cap E_0| = j > p-j-1 = |C \cap E_p|$, a contradiction for the minimality condition again.
Thus, $E_j$ is one of the disks in the first half of the sequence, that is, $1 \leq j < p/2$.

\smallskip

{\noindent \sc Claim.}
The disk $E_j$ is one of $E_1$, $E_{q'-1}$ or $E_{q'}$, where $q'$ is the unique integer satisfying $1 \leq q' \leq p/2$ and $qq' \equiv \pm 1 \pmod{p}$.

\smallskip

{\noindent \it Proof of Claim.}
We have assumed that $p \equiv \pm 1 \pmod{q}$ with $1 \leq q \leq p/2$, and so $q' = 1$ if $q = 1$, and $p= qq' + 1$ if $q=2$, and $p = qq' \pm 1$ if $q \geq 3$.
Assigning symbols $x$ and $y$ to oriented $\partial E'$ and $\partial E_0'$ respectively, $\partial E_{q'}$ may represent the primitive element of the form $xy^qxy^q \cdots xy^qxy^{q \pm 1}$ if $q \geq 2$ or $xy^p$ if $q = 1$.
In general, $\partial E_k$ represents an element of the form $xy^{n_1}xy^{n_2} \cdots xy^{n_k}$ for some positive integers $n_1, \cdots, n_k$ with $n_1+ \cdots + n_k = p$ for each $k \in \{1, 2, \cdots, p\}$.
Furthermore, since $C$ is disjoint from $E_j$ and $E_{j+1}$, the word determined by the arc $\partial C \cap \Sigma_j$ is of the form $y^{m_1}xy^{m_2} \cdots xy^{m_{j+1}}$ (or its reverse) when $\partial E_{j+1}$ represents an element of the form $xy^{m_1}xy^{m_2} \cdots xy^{m_{j+1}}$.

If $2 \leq j \leq q'-2$, then an element represented by $\partial E_{j+1}$ has the form $xy^qxy^q \cdots xy^q xy^{p-jq}$, and so an element represented by $\partial D$ contains $xy^qx$ and $y^{p-jq}$, which lies in the part $\partial C \cap \Sigma_j$ of $\partial D$.
We have $q' \geq 4$ in this case, and so $q \geq 2$.
Thus $p-jq = qq' \pm 1 - jq \geq q+2$.
By Lemma \ref{lem:key}, the disk $D$ cannot be primitive, a contradiction.

Suppose that $q' < j <p/2$.
First, observe that $\partial E_{q'+1}$ represents an element of the form $xy^q \cdots xy^qxy$ if $p = qq'+1$ or $xyxy^{q-1}xy^q \cdots xy^qxy^{q-1}$ if $p=qq' -1$.
Also a word represented by $\partial E_{j+1}$ is obtained by changing one $xy^q$ of a word represented by $\partial E_j$ into $xy^{q-1}xy$ or $xyxy^{q-1}$.
Thus, when we write $xy^{n_1}xy^{n_2} \cdots xy^{n_{j+1}}$ a word represented by $\partial E_{j+1}$, at least one of $n_2, n_3, \cdots, n_j$ must be $1$, and one of $n_1, n_2, \cdots, n_{j+1}$ is greater than $2$.
Since $C$ is disjoint from $E_j$ and $E_{j+1}$, the word corresponding to $\partial C \cap \Sigma_j$ is of the form $y^{n_1}xy^{n_2} \cdots xy^{n_{j+1}}$, which contains both of $xyx$ and $y^n$ for some $n > 2$.
Consequently, by Lemma \ref{lem:key}, the disk $D$ cannot be primitive, a contradiction again.

\smallskip

From the claim, at least one of the disks from surgery on $E$ along $C$ is either $E_1$ or $E_{q'}$.
The disk $E_1$ is primitive, and since we assumed that the circle $\partial E'_0$ is a $(p, q)$-curve on the boundary of the solid torus $\cl(V - \Nbd(E \cup E'))$, the disk $E_{q'}$ is also primitive by Lemma \ref{lem:sequence}, which completes the proof.
\end{proof}

In the proof of the above theorem, we assumed $\bar q = q$, which implied that a resulting disk from surgery is $E_1$ or $E_{q'}$.
The same result holds when $\bar q = p-q$.
But if we assume $\bar q = \{q', p-q'\}$, then the resulting disk will be $E_1$ or $E_q$ which turn out to be primitive in the corresponding sequence of disks.
Together with this observation, assuming that $D$ is disjoint from $E$, and so taking the disk $D$ instead of an outermost subdisk $C$, we have the following result.

\begin{lemma}
Let $\{D, E\}$ be a primitive pair of $V$.
Then there is a sequence of disks $E_0, E_1, \cdots, E_p$ generated by a dual pair $\{E, E'\}$ for some dual disk $E'$ of $E$, such that $D$ is one of the disks $E_1$, $E_q$ or $E_{q'}$ in the sequence, where $q'$ is the unique integer satisfying $qq' \equiv \pm 1 \pmod{p}$ with $1 \leq q' \leq p/2$.
\label{lem:primitive_pair}
\end{lemma}

We say simply that a primitive pair has a common dual disk if the two disks of the pair have a common dual disk.

\begin{theorem}
Given a lens space $L(p, q)$ with $1 \leq q \leq p/2$, each primitive pair has a common dual disk if and only if $q = 1$.
In this case, it has a unique common dual disk if $p \geq 3$, and has  exactly two disjoint common dual disks if $p = 2$.
\label{thm:common_dual}
\end{theorem}

\begin{proof}
Suppose that $q = 1$, and let $\{D, E\}$ be any primitive pair of $V$.
By Lemma \ref{lem:primitive_pair}, there is a sequence of disks $E_0, E_1, \cdots, E_p$ generated by a dual pair $\{E, E'\}$, in which $D$ is $E_1$ (here we have $q' = q = 1$).
By Lemma \ref{lem:common_dual}, $D$ and $E$ have a common dual disk.

\begin{figure}
\centering
\includegraphics[width=11.5cm,clip]{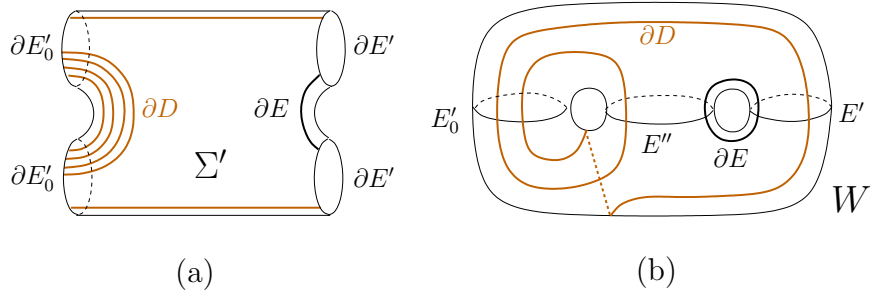}
\caption{(a) $\partial E$ and $\partial D$ lying in the $4$-holed sphere $\Sigma'$ (when $p = 5$ for example).
(b) Two common dual disks $E'$ and $E''$ of $D$ and $E$ for $L(2, 1)$.}
\label{common_dual}
\end{figure}

Now, let $E'$ be a common dual disk of $D$ and $E$.
Let $E'_0$ be the unique semiprimitive disk in $W$ disjoint from $E \cup E'$.
We recall that $E_0'$ is the meridian disk of the solid torus $\cl(W - \Nbd(E'))$.
Then $\partial E'_0$ intersects $\partial D$ in $p$ points.
Cut the surface $\partial W$ along the boundary circles $\partial E'$ and $\partial E'_0$ to obtain the $4$-holed sphere $\Sigma'$.
In $\Sigma'$, the boundary circle $\partial E$ is a single arc connecting two boundary circles of $\Sigma'$ that came from $\partial E'$.
But the boundary circle $\partial D$ in $\Sigma'$ consists of $(p-1)$ arcs connecting two boundary circles that came from $\partial E'_0$ together with two arcs connecting $\partial E'$ and $\partial E'_0$ as in Figure \ref{common_dual} (a).
Observe that if there is a common dual disk of $D$ and $E$ other than $E'$, then it cannot intersect $E' \cup E'_0$ otherwise it intersects $\partial D$ or $\partial E$ in more than one points.
Thus the boundary of any common dual disk $E''$ of $D$ and $E$ other than $E'$ is a circle inside $\Sigma'$, and hence, from the figure, it is obvious that one more common dual disk $E''$ other than $E'$ exists if and only if $p = 2$, and such an $E''$ is unique in this case.
See Figure \ref{common_dual} (b).

Conversely, suppose that every primitive pair has a common dual disk. Choose any sequence of disks $E_0, E_1, \cdots, E_p$ in $V$ generated by some dual pair $\{E, E'\}$.
Then one of the disks $E_{q'}$ and $E_q$ in the sequence is primitive, where $q'$ satisfies $1 \leq q' \leq p/2$ and $qq' = \pm 1 \pmod{p}$, which forms a primitive pair with $E$.
If $\{E, E_{q'}\}$ is a primitive pair, then it has a common dual disk, and so, by Lemma \ref{lem:common_dual}, there is a semiprimitive disk in $V$ disjoint from $E$ and $E_{q'}$.
The only possible semiprimitive disk disjoint from $E$ and $E_{q'}$ is $E_{q'-1}$ or $E_{q'+1}$ by Lemma \ref{lem:surgery_on_primitive}, that is, $E_{q' - 1}$ is $E_0$ or $E_{q'+1} = E_p$.
In any cases, we have $q = 1$ (the latter case implies $(p, q) = (2, 1)$).
The same conclusion holds in the case where $\{E, E_q\}$ is a primitive pair.
\end{proof}

It is clear that any primitive disk is a member of infinitely many primitive pairs.
But a primitive pair can be contained at most two primitive triples, which is shown as follows.

\begin{theorem}
Given a lens space $L(p, q)$, for $1 \leq q \leq p/2$, with a genus two Heegaard splitting $(V, W; \Sigma)$ of $L(p, q)$, there is a primitive triple in $V$ if and only if $q = 2$ or $p= 2q +1$.
In this case, we have the following refinements.
\begin{enumerate}
\item If $p = 3$, then each primitive pair is contained in a unique primitive triple.
    Further, each of the three primitive pairs in any primitive triple have unique common dual disks, which form a primitive triple of $W$.
\item If $p = 5$, then each primitive pair having a common dual disk is contained in a unique primitive triple, and each having no common dual disk is contained in exactly two primitive triples.
\item If $p \geq 7$, then each primitive pair having a common dual disk is contained either in a unique or in no primitive triple, and each having no common dual disk is contained in a unique primitive triple.
\item If $p \geq 5$, then exactly one of the three primitive pairs in any primitive triple has a common dual disk.
\end{enumerate}
\label{thm:triple}
\end{theorem}

\begin{proof}
Note that $L(2q + 1, q)$ is homeomorphic to $L(2q+1, 2)$.
We prove first the necessity together with the refinements.
Suppose that $q=2$ or $p=2q+1$, and let $\{D, E\}$ be any primitive pair of $V$.
By Lemma \ref{lem:primitive_pair}, there is a sequence of disks $E_0, E_1, \cdots, E_p$ generated by a dual pair $\{E, E'\}$, in which $D$ is one of $E_1$, $E_2$ or $E_q$.

If $p=3$, the disk $D$ is $E_1$, and so $E_2$ is the unique primitive disk disjoint from $E \cup E_1$ by Lemma \ref{lem:surgery_on_primitive}.
Thus $\{D, E\}$ is contained in the unique primitive triple $\{D, E, E_2\}$.
The primitive pairs $\{E, D\} = \{E, E_1\}$ and $\{E, E_2\}$ in the triple have unique common dual disks by Lemma \ref{lem:common_dual} and Theorem \ref{thm:common_dual}.
Figure \ref{triple} (a) illustrates the boundary circles of the common dual disks $E'$ and $E''$ of $\{E, E_1\}$ and $\{E, E_2\}$ respectively.
In the figure, $\partial E'$ is the union of two arcs connecting the points $a$ and $b$, and $\partial E''$ is the union of three arcs connecting $c$, $d$ and $e$.
The circle $\partial E''$ really bounds a disk in $W$ since it lies in the boundary of the $3$-ball $W$ cut off by $E' \cup E'_0$, where $\partial E'_0$ is the union of the dotted arcs in the figure.
Observe that $E'$ and $E''$ are disjoint from and are not isotopic to each other.

Furthermore, exchanging the roles of $D$ and $E$, we have the sequence of disks $D_0, D_1, D_2, D_3$ generated by $\{D, E'\}$ where $D_0 = E_0$, $D_1 = E$ and $D_2 = E_2$.
By the same argument to the above, there exists the unique common dual disk $E'''$ of $\{D, D_2\} = \{D, E_2\}$ which is disjoint from and is not isotopic to the common dual disk $E'$ of $\{D, D_1\} = \{D, E\}$.
Considering one more sequence of disks, that is, generated by $\{E_2, E''\}$, whose first two disks are $E_3$ and $E$, we see that $E'''$ is disjoint from and is not isotopic to $E''$, and hence $\{E', E'', E'''\}$ is a primitive triple in $W$.

\begin{figure}
\centering
\includegraphics[width=12.5cm,clip]{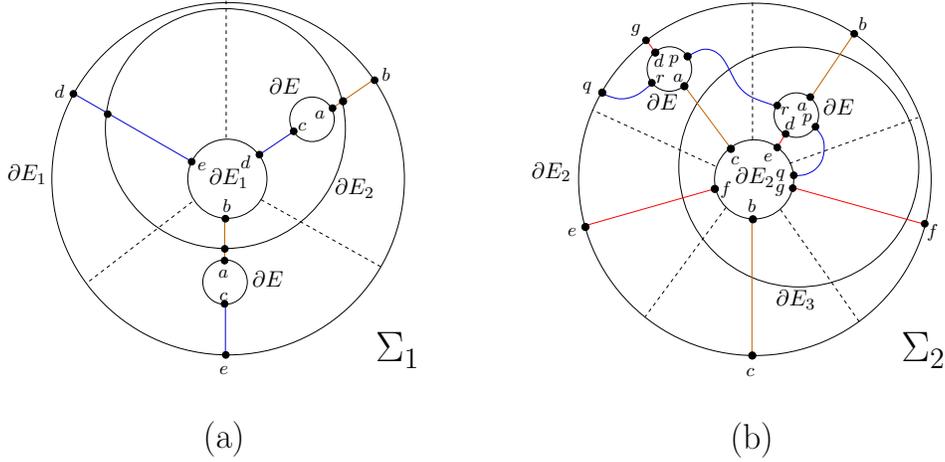}
\caption{(a) The $4$-holed sphere $\Sigma_1$ in $L(3, 1)$.
(b) The $4$-holed sphere $\Sigma_2$ in $L(5, 2)$.}
\label{triple}
\end{figure}

If $p=5$, then the disk $D$ is either $E_1$ or $E_2$.
If $\{D, E\}$ has a common dual disk, then $D$ is $E_1$, and they are contained in the unique primitive triple $\{D, E, E_2\}$.
If $\{D, E\}$ has no common dual disk, then $D$ is $E_2$, and they are contained in exactly two primitive triples $\{D, E, E_1\}$ and $\{D, E, E_3\}$.

If $p \geq 7$, then $D$ is either $E_1$, $E_2$ or $E_q$.
Observe that if one of $E_2$ and $E_q$ are primitive in the sequence, then the other one is not, while $E_1$ is always primitive.
If $\{D, E\}$ has no common dual disk, then $D$ is $E_2$ or $E_q$.
In this case, $\{D, E\}$ is contained in the unique primitive triple $\{D, E, E_1\}$ if $D$ is $E_2$, or in the unique triple $\{D, E, E_{q+1}\}$ if $D$ is $E_q$.
Suppose next that $\{D, E\}$ has a common dual disk.
Then $D$ is $E_1$, and hence $\{D, E\}$ is either contained in a unique primitive triple or contained in no primitive triple, according as $E_2$ is primitive or not.

Now suppose that $p \geq 5$, and let $\{D, E, F\}$ be any primitive triple of $V$.
Consider again the sequence of disks $E_0, E_1, \cdots, E_p$ generated by a dual pair $\{E, E'\}$ in the above, where $D$ is one of $E_1$, $E_2$ or $E_q$.
We have only two possibilities, the triple $\{D, E, F\}$ is $\{E, E_1, E_2\}$ or $\{E, E_q, E_{q+1}\}$.
In case of $\{D, E, F\} = \{E, E_1, E_2\}$, the pair $\{E, E_1\}$ has a unique common dual disk, say $E'$, while $\{E, E_2\}$ has not.
For the pair $\{E_1, E_2\}$, exchanging the roles of $E_1$ and $E$, we have the sequence of disks $D_0, D_1, \cdots, D_p$ generated by the dual pair $\{E_1, E'\}$, where $D_0 = E_0$, $D_1 = E$ and $D_2 = E_2$.
Thus $\{E_1, E_2\} = \{E_1, D_2\}$ has no common dual disk.

In case of $\{D, E, F\} = \{E, E_q, E_{q+1}\}$, it is obvious that each of $\{E, E_q\}$ and $\{E, E_{q+1}\}$ has no common dual disks
while $\{E_q, E_{q+1}\}$ has one.
The Figure \ref{triple} (b) illustrates the case of $L(5, 2)$, and so $\{E_q, E_{q+1}\} = \{E_2, E_3\}$.
In the figure, the union of three arcs $p$, $q$ and $r$ is a circle, say $l$, which intersects each of $\partial E_2$ and $\partial E_3$ in a single point.
To see that $l$ bounds a disk in $W$, consider first the circle $m$ that is the union of three arcs connecting $d, e, f$ and $g$.
The circle $m$ bounds a disk in $W$ since it lies in the boundary of the $3$-ball $W$ cut off by $E' \cup E'_0$, where $\partial E'$ is the union of three arcs connecting $a, b$ and $c$, and $\partial E'_0$ is the union of the five dotted arcs in the figure.
Consequently, we see that $l$ also bounds a disk in $W$, that is the common dual disk of $\{E_2, E_3\}$, since $l$ lies in the boundary of the $3$-ball $W$ cut off by $E'$ and the disk bounded be $m$.
Note that Figure \ref{triple} (b) can be adapted easily for the pair $\{E_q, E_{q+1}\}$ in any $L(p, 2)$ with $p \geq 5$ to show that the pair has a common dual disk.

Conversely, suppose that there is a primitive triple $\{D, E, F\}$ of $V$.
Choose a dual disk $E'$ of $E$ so that the resulting semiprimitive disk $E_0$ intersects $D$ minimally.
If $D$ is disjoint from $E_0$, take $E_1 = D$, and if $D$ intersects $E_0$, take $E_1$ the disk obtained from surgery on $E_0$ along an outermost subdisk of $D$ cut off by $D \cap E_0$ other than $E$.
Then we have the sequence of disks $E_0, E_1, \cdots, E_p$ generated by the dual pair $\{E, E'\}$ starting with $E_0$ and $E_1$.
The boundary circle of the meridian disk of the solid torus $\cl(W - \Nbd(E \cup E'))$ is a $(p, \bar q)$-curve for some $\bar q \in \{q, q', p-q', p-q\}$ on the boundary of $\cl(V - \Nbd(E \cup E'))$, where $q'$ is the unique integer satisfying $qq' \equiv \pm 1$ (mod $p$) and $1 \leq q' \leq p/2$.
Considering all the possible $(p, \bar q)$-curves, the disk $D$ must be one of $E_1$, $E_{q'}$ and $E_q$.
In each case, the primitive disk $F$ is a disk in the sequence which is adjacent to $D$.
Thus $\{D, F\}$ is one of $\{E_1, E_2\}$, $\{E_q, E_{q+1}\}$ and $\{E_{q'}, E_{q'+1}\}$, which implies $q=2$ or $p= 2q+1$.
\end{proof}

\section{The disk complexes}
\label{sec:disk_complex}

Let $M$ be an irreducible $3$-manifold with compressible boundary.
The {\it disk complex} of $M$ is a simplicial complex defined as follows.
The vertices are the isotopy classes of essential disks in $M$, and a collection of $k+1$ vertices spans a $k$-simplex if and only if it admits a collection of representative disks which
are pairwise disjoint.
In particular, if $M$ is a handlebody of genus $g \geq 2$, then the disk complex is $(3g - 4)$-dimensional and is not locally finite.
The following is a key property of a disk complex.

\begin{theorem}
If $\mathcal K$ is a full subcomplex of the disk complex satisfying the following condition, then $\mathcal K$ is contractible.

\begin{itemize}
 \item Let $E$ and $D$ be disks in $M$ representing vertices of $\mathcal K$.
 If they intersect each other transversely and minimally, then at least one of the disks from surgery on $E$ along an outermost subdisk of $D$ cut off by $D \cap E$ represents a vertex of $\mathcal K$.
\end{itemize}
\label{thm:surgery}
\end{theorem}

In \cite{C}, the above theorem is proved in the case where $M$ is a handlebody, but the proof is still valid for an arbitrary irreducible manifold with compressible boundary. From the theorem, we see that the disk complex itself is contractible, and the {\it non-separating disk complex} is also contractible, which is the full subcomplex spanned by the vertices of non-separating disks.
We denote by $\mathcal D(M)$ the non-separating disk complex of $M$.

Consider the case that $M$ is a genus two handlebody $V$.
Then the complex $\mathcal D(V)$ is $2$-dimensional, and every edge of $\mathcal D(V)$ is contained in infinitely but countably many $2$-simplices.
For any two non-separating disks in $V$ which intersect each other transversely and minimally, it is easy to see that ``both'' of the two disks obtained from surgery on one along an outermost subdisk of another cut off by their intersection are non-separating.
\begin{figure}
\centering
\includegraphics[width=7cm,clip]{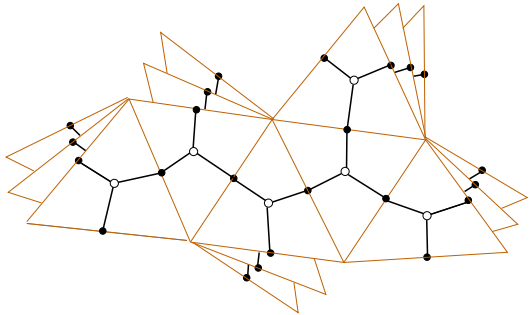}
\caption{A portion of the non-separating disk complex $\mathcal D(V)$ of a genus two handlebody $V$ with its dual complex, a tree.}
\label{disk_complex}
\end{figure}
This implies, from Theorem \ref{thm:surgery}, that $\mathcal D(V)$ and the link of any vertex of $\mathcal D(V)$ are all contractible.
Thus we see that $\mathcal D(V)$ deformation retracts to a tree in the barycentric subdivision of it.
Actually, this tree is a dual complex of $\mathcal D(V)$.
A portion of the non-separating disk complex of $V$ together with its dual tree is described  in Figure \ref{disk_complex}.

\section{The primitive disk complexes}
\label{sec:primitive_disk_complex}

Now we return to the genus two Heegaard splitting $(V, W; \Sigma)$ of a lens space $L = L(p, q)$.
Again we will assume that $1 \leq q \leq p/2$.
The {\it primitive disk complex} $\mathcal P(V)$ for $L(p, q)$ is defined to be the full subcomplex of $\mathcal D(V)$ spanned by the vertices of primitive disks in $V$.
From the structure of $\mathcal D(V)$, we see that every connected component of any full subcomplex of $\mathcal D(V)$ is contractible.
Thus $\mathcal P(V)$ is either contractible by itself or each of its connected components is contractible.
In this section, we describe the complete combinatorial structure of the primitive disk complex $\mathcal P(V)$ for each lens space. In particular, we find all lens spaces whose primitive disk complexes are contractible.

As in the previous sections, let $E$ be a primitive disk in $V$ with a dual disk $E'$.
The disk $E'$ forms a complete meridian system of $W$ together with the semiprimitive disk $E'_0$ in $w$ disjoint from $E \cup E'$.
Assigning the symbols $x$ and $y$ to the oriented circles $\partial E'$ and $\partial E'_0$ respectively, any oriented simple closed curve, especially the boundary circle of any essential disk in $V$, represents an element of the free group $\pi_1(W) = \langle x, y \rangle$ in terms of $x$ and $y$.

Let $D$ be a non-separating disk in $V$.
A simple closed curve $l$ on $\partial V$
intersecting $\partial D$ transversely in a
single point is called a {\it dual circle} of $D$.

\begin{lemma}
\label{lem:key lemma for non-connectivity}
Let $\{D_1, D_2\}$ be a complete meridian system of $V$.
Suppose that the non-separating disks $D_1$ and $D_2$ satisfy the following conditions:
\begin{enumerate}
\item
for each $i \in \{ 1,2 \}$, all intersections of $\partial D_i$
and $\partial E^\prime$ have the same sign;
\item
for each $i \in \{ 1 , 2 \}$, the circle $\partial D_i$ represents an element $w_i$ of the form
$(xy^q)^{m_i} x y^{n_i}$, where
$0 \leqslant m_1 < m_2$ and $n_1 \neq n_2$;
\item
any subarc of $\partial E^\prime$ with both endpoints on $\partial D_1$
intersects $\partial D_2$; and
\item
there exists a dual circle of $D_1$ disjoint from
$\partial D_2 \cup \partial E^\prime$.
\end{enumerate}
Then there exists a non-separating disk $D_*$ in $V$
disjoint from $D_1 \cup D_2$ satisfying the following:
\begin{enumerate}
\item
all intersections of $\partial D_*$
and $\partial E^\prime$ have the same sign;
\item
$\partial D_*$ represents an element of the form
$(xy^q)^{m_1 + m_2 + 1} x y^{n_1 + n_2 - q}$;
\item
for each $i \in \{1 , 2 \}$, any subarc of
$\partial E^\prime$ with both endpoints on $\partial D_i$
intersects $\partial D_*$; and
\item
for each $i \in \{1 , 2 \}$,
there exists a dual circle of $D_i$ disjoint from
$\partial D_* \cup \partial E^\prime$.
\end{enumerate}
\end{lemma}

\begin{proof}
For $i \in \{1,2\}$, let $\nu_i$
be a connected subarc of
$\partial D_i$
that determines the subword
$y^{n_i}$ of $w_i$.
Cutting off $\partial V$ by $\partial D_1 \cup \partial D_2$,
we obtain the 4-holed sphere $\Sigma_\ast$.
We denote by $d_i^\pm$ the boundary circles
of $\Sigma_\ast$ coming from $\partial D_i$, and
by $\nu_i^\pm$ the subarc of $d_i^\pm$
coming from $\nu_i$.
By the assumption (2),
we may assume without loss of generality that
each oriented arc component
$\partial E^\prime \cap \Sigma_\ast$
directs from $d_{i_1}^+$ to $d_{i_2}^-$ for certain
$i_1, i_2 \in \{ 1 , 2 \}$.
By the assumptions (3) and (4),
the 4-holed sphere $\Sigma_\ast$ and the arcs $\Sigma_\ast \cap \partial E^\prime$
on $\Sigma_\ast$ can be drawn as in one of the
two diagrams of Figure \ref{fig:L_or_R-replacement}.
\begin{figure}
\centering
\includegraphics[width=12.5cm,clip]{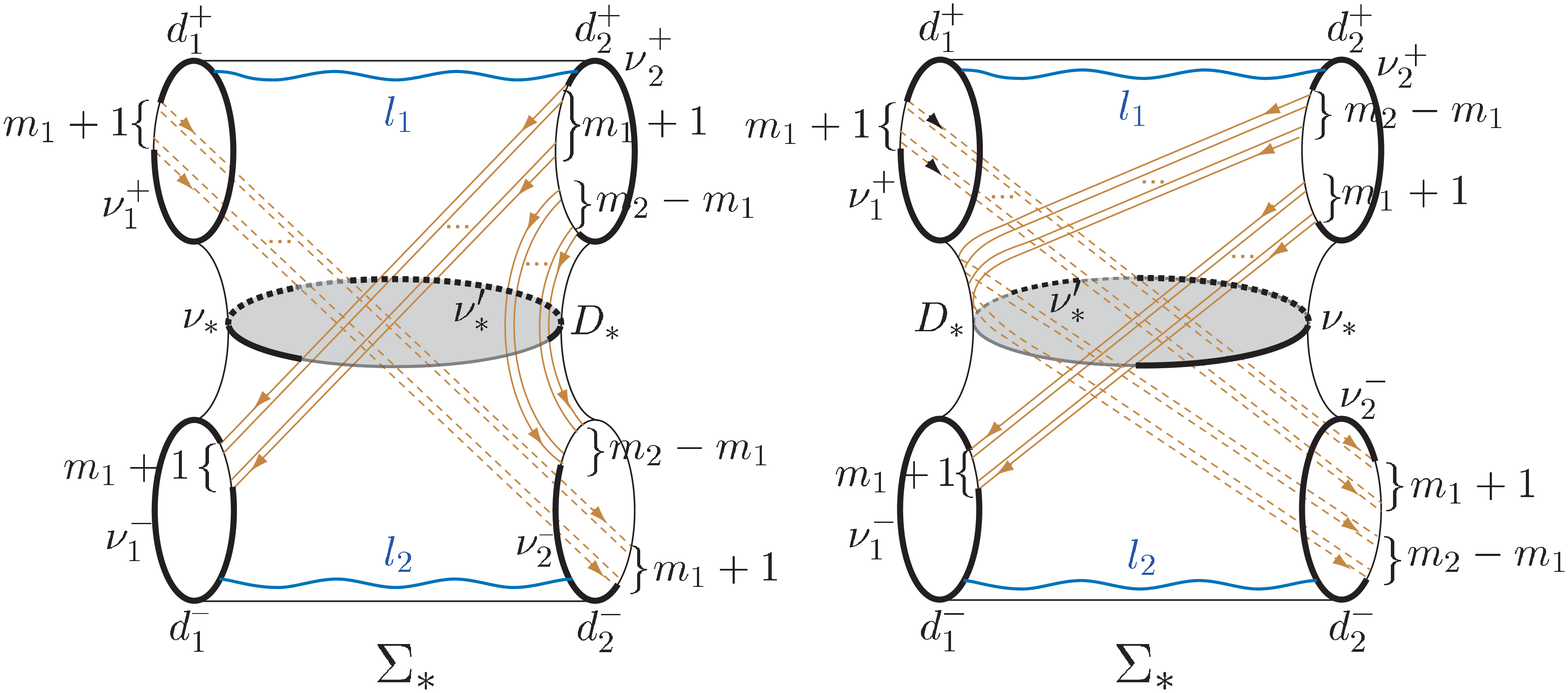}
\caption{}
\label{fig:L_or_R-replacement}
\end{figure}
In the figure the arcs $\nu_i^{\pm}$ in $d^{\pm}_i$ are
drawn in bold.

Let $D_*$ be the horizontal disk in each of the diagrams.
It is clear that
$D_*$ satisfies conditions (1) and (3).
Also, the simple closed curve on $\partial V$
obtained from the two arcs
$l_1$ and $l_2$ depicted in the figure
by gluing back along $d_1^{\pm}$ and $d_2^\pm$
is a dual circle of $D_i$ disjoint from
$\partial D_* \cup \partial E^\prime$, and so the condition (4) holds for $D_*$.
Moreover, it is easily seen that
all but one component $\nu_*$ of
$\partial D_*$ cut off by $\partial E^\prime$,
shown in Figure \ref{fig:L_or_R-replacement},
determine a word of the form $y^q$.
Hence it suffices to show that
the arc $\nu_*$ determines a word of the
form $y^{n_1 + n_2 - q}$.
From the arcs $\nu_1^+ \cup \nu_2^+$, algebraically
$n_1 + n_2$ arcs of $\partial E_0^\prime \cap \Sigma_\ast$
come down and all of them pass trough
$\nu_* \cup \nu_*^\prime$ from above,
where the arc $\nu_*^\prime$ is shown in
Figure \ref{fig:L_or_R-replacement}.
Since the arc $\nu_*^\prime$ determines a word of the
form $y^q$,
the arc $\nu_*$ determines a word of the
form $y^{n_1 + n_2 - q}$.
\end{proof}

Let $(D_1, D_2)$ be an ordered pair of
disjoint non-separating disks in $V$
such that the (unordered) pair $\{D_1, D_2 \}$
satisfies the conditions of Lemma
\ref{lem:key lemma for non-connectivity}.
Then there exists a disk $D_*$
as in the lemma
and we again obtain new ordered pairs
$(D_1, D_*)$ and $(D_*, D_2)$ such that
both $\{ D_1, D_* \}$ and $\{ D_*, D_2 \}$
satisty the conditions of the lemma.
We call these new pairs $(D_1 , D_*)$ and
$(D_* , D_2)$ the pairs obtained by
R-replacement and L-replacement, respectively,
of $(D_1, D_2)$.

\vspace{1em}

\begin{theorem}
For a lens space $L(p, q)$ with $1 \leq q \leq p/2$,
the primitive disk complex $\mathcal P(V)$ for $L(p, q)$ is contractible if and only if $p \equiv \pm 1 \pmod{q}$.
\label{thm:contractible}
\end{theorem}

\begin{proof}
The sufficiency follows directly from Theorem \ref{thm:primitive} and Theorem \ref{thm:surgery}.
For the necessity, we will show that $\mathcal P(V)$ is not connected.
Suppose that $p$ and $q$ do not satisfy the condition $p \equiv \pm 1 \pmod{q}$.
Then there exist unique relatively prime integers $m$ and $r$ such that $p=qm + r$ with $2 \leq r \leq q-2$, and hence there exist natural numbers $s$ and $t$ with
$s r - t q = q + 1$.
Consider the unique continued fraction expansion
\[
s / (t+1) =
p_0 +  \frac{1}{p_1 + \frac{1}{p_2 + \frac{1}{\ddots + \frac{1}{p_k}}}} \]
where $p_i \geq 1$ for $i \in \{0, 1, \cdots, k-1\}$ and $p_k \geq 2$.

Let $E$ be any primitive disk in $V$ and let $E'$ be a dual disk of $E$.
The boundary circle of the semiprimitive disk in $W$ disjoint from $E \cup E'$ is a $(p, \bar q)$-curve on the boundary of the solid torus $\cl(V - \Nbd(E \cup E'))$ for some $\bar q \in \{q, q', p-q', p-q\}$, where $q'$ satisfies $0 < q' < p$ and $qq' \equiv \pm 1 \pmod p$.
We will consider only the case of $\bar q = q$, that is, the boundary circle is a $(p, q)$-curve.
The following argument can be easily adapted for the other cases.

Consider any sequence of disks $E_0, E_1, \cdots, E_p$ in $V$ generated by a dual pair $\{E, E'\}$.
Note that the disks $E_m$ and $E_{m+1}$ in the sequence are not primitive since $\partial E_m$ and $\partial E_{m+1}$ represent elements of the form $(xy^q)^{m-1}xy^{q + r}$ and $(xy^q)^m xy^r$ respectively.
Set $D_0 = E_m$ and $D_{-1} = E$.
Since $D_0$ and $D_{-1}$ satisfy the conditions of
Lemma \ref{lem:key lemma for non-connectivity},
we obtain a new ordered pair $(D_0 , D_1)$ by an
R-replacement of $(D_0, D_{-1})$.
The disk $D_1$ is not primitive since $\partial D_1$ represents an element of the form
$(xy^q)^m xy^r$.
(Actually, $D_1$ can be chosen to be the disk $E_{m+1}$ in the sequence.)
Applying R-replacements $(p_0 - 1)$ times more, starting at
$(D_0 , D_1)$, as
\[ (D_0, D_1) \to (D_0, D_2) \to \cdots  \to (D_0, D_{p_0}) , \]
we obtain the pair $(D_0 , D_{p_0})$.
Next we apply L-replacements $p_1$ times starting at $(D_0 , D_{p_0})$ as
\[ (D_0, D_{p_0}) \to (D_{p_0 + 1}, D_{p_0}) \to (D_{p_0 + 2}, D_{p_0})
\to \cdots  \to (D_{p_0 + p_1}, D_{p_0}) \]
to obtain the pair $(D_{p_0 + p_1}, D_{p_0})$.
Continuing this process, we finally obtain either the pair
$(D_{p_0 + \cdots + p_k}, D_{p_0 + \cdots + p_{k-1}})$
if $k$ is odd, or
$(D_{p_0 + \cdots + p_{k-1}}, D_{p_0 + \cdots + p_k})$
if $k$ is even, of pairwise disjoint non-separating disks.
See Figure \ref{fig:sequence_of_triangles}.

\begin{figure}
\centering
\includegraphics[width=12.5cm,clip]{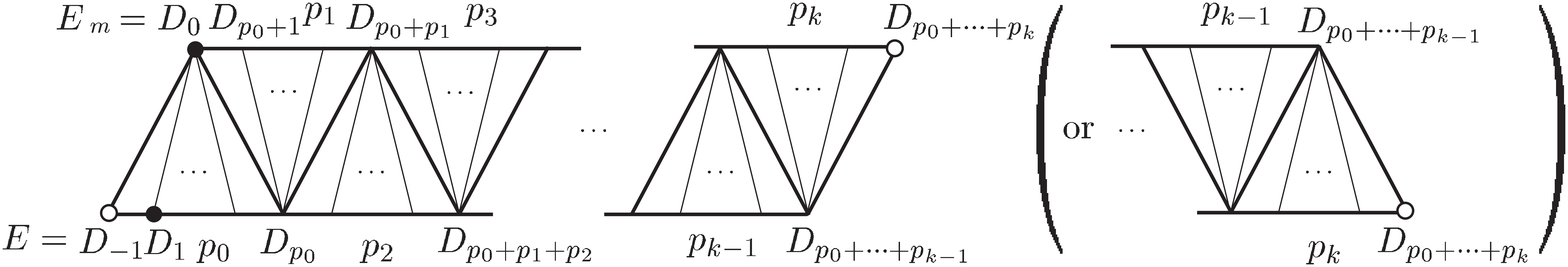}
\caption{The portion of $\mathcal{D}(V)$ obtained by
L and R-replacements from $(D_0, D_{-1})$.}
\label{fig:sequence_of_triangles}
\end{figure}

We assign $D_0$ and $D_{-1}$ the rational numbers
$1/0$ and $0/1$, respectively.
We inductively assign rational numbers to the disks appearing in the above process as follows.
Let $(D_* , D_{**})$ be an ordered pair of non-separating disks
appearing in the process.
Assume that we have already assigned $D_*$ and $D_{**}$
rational numbers $a_1 / b_1$ and $a_2 / b_2$, respectively.
Then we assign the next disk obtained by
L or R-replacement of $(D_*, D_{**})$
the rational number $(a_1 + a_2) / (b_1 + b_2)$.

\smallskip

\noindent {\it Claim.}
If a disk $D_j$, for $-1 \leq j \leq p_0 + \cdots + p_k$,
appearing in the above process is assigned
a rational number $a / b$, then $\partial D_j$
represents an element of the form
$(xy^q)^{d} xy^{a r - (b-1)q}$ for some non-negative integer $d$.

\smallskip

\noindent {\it Proof of Claim.}
If $j = -1$, then $a / b = 0 / 1$ and
$\partial D_{-1} = \partial E$ represents
$x$ and we have $a r - (b - 1)q = 0$.
If $j = 0$, then $a / b = 1 / 0$ and
$\partial D_{0} = \partial E_m$ represents an element of the form
$(xy^q)^{m - 1}xy^{q+r}$ and we have $a r - (b - 1)q = q + r$.

Assume that the claim is true for any $D_i$ with
$i$ less than $j$ and that $D_j$ is obtained from $(D_* , D_{**})$.
If $D_*$ and $D_{**}$
are assigned rational numbers $a_1 / b_1$ and $a_2 / b_2$,
respectively, then $D_j$ is assigned $(a_1 + a_2) / (b_1 + b_2)$ by definition.
By the assumption, $\partial D_*$ and $\partial D_{**}$
determine elements of the forms
$(xy^q)^{d_1} xy^{a_1 r - (b_1-1)q}$ and
$(xy^q)^{d_2} xy^{a_2 r - (b_2-1)q}$ respectively,
for some non-negative integers $d_1$ and $d_2$.
By Lemma \ref{lem:key lemma for non-connectivity},
the circle $\partial D_j$ determines an element of the form
$(xy^q)^{d_1 + d_2 + 1} xy^{(a_1+a_2) r - (b_1+b_2-1)q}$, and hence the induction completes the proof.

\smallskip

Due to well-known properties of the Farey graph, see e.g. \cite{HT85},
$D_{p_0 + \cdots + p_k}$ is assigned $s / (t+1)$.
Therefore, by the claim,
$\partial D_{p_0 + \cdots + p_k}$
determines an element of the form
$(xy^q)^d xy^{sr - tq}$, hence
$(xy^q)^d xy^{q+1}$.
This implies that $D_{p_0 + \cdots + p_k}$ is primitive.

Now, we focus on the four disks
$D_{-1}$, $D_0$, $D_1$ and $D_{p_0 + \cdots + p_k}$.
Since the dual complex of the disk complex
$\mathcal{D}(V)$ is a tree, and
the disks $D_0$ and $D_1$ are not primitive,
the primitive disks $D_{-1}$ and $D_{p_0 + \cdots + p_k}$
belong to different components of $\mathcal{P}(V)$.
This implies that $\mathcal{P}(V)$ is not connected.
\end{proof}

Now we are ready to give a precise description of the primitive disk complex $\mathcal P(V)$ of each lens space.
For convenience, we classify all the edges and $2$-simplices of $\mathcal P(V)$ as follows.

\begin{enumerate}
\item An edge of $\mathcal P(V)$ is called an {\it edge of type-$0$} ({\it type-$1$, type-$2$,} respectively) if a primitive pair representing the end vertices of the edge has no common dual disk (has a unique common dual disk, has exactly two common dual disks which form a primitive pair in $W$, respectively).
\item A $2$-simplex of $\mathcal P(V)$ is called a {\it $2$-simplex of type-$1$} ({\it of type-$3$,} respectively) if exactly one of the three primitive pairs in the primitive triple representing the three edges of the $2$-simplex has a unique common dual disk (if all the three pairs have unique common dual disks which form a primitive triple in $W$, respectively).
\end{enumerate}

By Theorems \ref{thm:common_dual} and \ref{thm:triple}, we see that each of the edges and $2$-simplices of $\mathcal P(V)$ is one of those types in the above.
Consider first the case where $\mathcal P(V)$ is not contractible, that is, non-connected.
In this case, $\mathcal P(V)$ is $1$-dimensional by Theorem \ref{thm:triple}.
Given a primitive disk $E$ in $V$, we can choose infinitely many non-isotopic dual disks $E'$ of $E$.
Considering the sequences of disks generated by the dual pair
$\{E, E'\}$ for each choice of $E'$, in the proof of Theorem \ref{thm:contractible}, we can find infinitely many connected components of $\mathcal P(V)$ which do not contain the vertex represented by $E$.
Thus $\mathcal P(V)$ has infinitely many connected components, which are all trees isomorphic to each other.
Any vertex of $\mathcal P(V)$ has infinite valency, and further, infinitely many edges of type-$0$ and of type-$1$ meet in each vertex.
The contractible case can be summarized as follows, which is a direct consequence of Theorems \ref{thm:contractible}, \ref{thm:common_dual} and \ref{thm:triple}.

\begin{theorem}
Given any lens space $L(p, q)$ with $1 \leq q \leq p/2$, if the primitive disk complex $\mathcal P(V)$ of $L(p, q)$ is contractible, then $p \equiv \pm 1 \pmod q$ and $\mathcal P(V)$ is contained in one of the following cases.
\begin{enumerate}
\item If $q \neq 2$ and $p \neq 2q + 1$, then $\mathcal P(V)$ is a tree, and every vertex has infinite valency.
    In this case,
    \begin{enumerate}
    \item if $p=2$ and $q=1$, then every edge is of type-$2$.
    \item if $p \geq 4$ and $q=1$, then every edge is of type-$1$.
    \item if $q \neq 1$, then every edge is of either type-$0$ or type-$1$, and infinitely many edges of type-$0$ and of type-$1$ meet in each vertex.
    \end{enumerate}
\item If $q = 2$ or $p=2q+1$, then $\mathcal P(V)$ is $2$-dimensional, and every vertex meets infinitely many $2$-simplices.
    In this case,
    \begin{enumerate}
    \item if $p = 3$, then every edge is of type-$1$, every $2$-simplex is of type-$3$, and every edge is contained in a unique $2$-simplex.
    \item if $p = 5$, then every edge is of either type-$0$ or type-$1$, and every $2$-simplex is of type-$1$. Every edge of type-$0$ is contained in exactly two $2$-simplices, while every edge of type-$1$ in a unique $2$-simplex.
    \item if $p \geq 7$, then every edge is of either type-$0$ or type-$1$, and every $2$-simplex is of type-$1$.
        Every edge of type-$0$ is contained in a unique $2$-simplex.
        Every edge of type-$1$ is contained in a unique $2$-simplex or in no $2$-simplex.
    \end{enumerate}
\end{enumerate}
\label{thm:structure}
\end{theorem}
Figure \ref{shape} illustrates a portion of each of the contractible primitive disk complexes $\mathcal P(V)$ classified in the above, together with its surroundings in $\mathcal D(V)$.
We label simply $E$ or $E_j$ for the vertices represented by disks $E$ or $E_j$.
In the case (2)-(b), the complex $\mathcal P(V)$ for $L(5, 2)$, every edge is contained a unique ``band''.
The edges in the boundary of a band are of type-$1$, while the edges inside a band are of type-$0$.
The whole figure of $\mathcal P(V)$ for $L(5, 2)$ can be imagined as the union of infinitely many bands such that any of two bands are disjoint from each other or intersects in a single vertex.
In the case (2)-(c), there are two kind of sequences of disks $E_0, E_1, \cdots, E_p$ generated by a dual pair $\{E, E'\}$.
The first one has primitive disks $E_1, E_q, E_{p-q}$ and $E_{p-1}$, while the second one has $E_1, E_2, E_{p-2}$ and $E_{p-1}$.
Figure \ref{shape} (2)-(c) illustrates an example of the first one.
\begin{figure}
\centering
\includegraphics[width=10cm,clip]{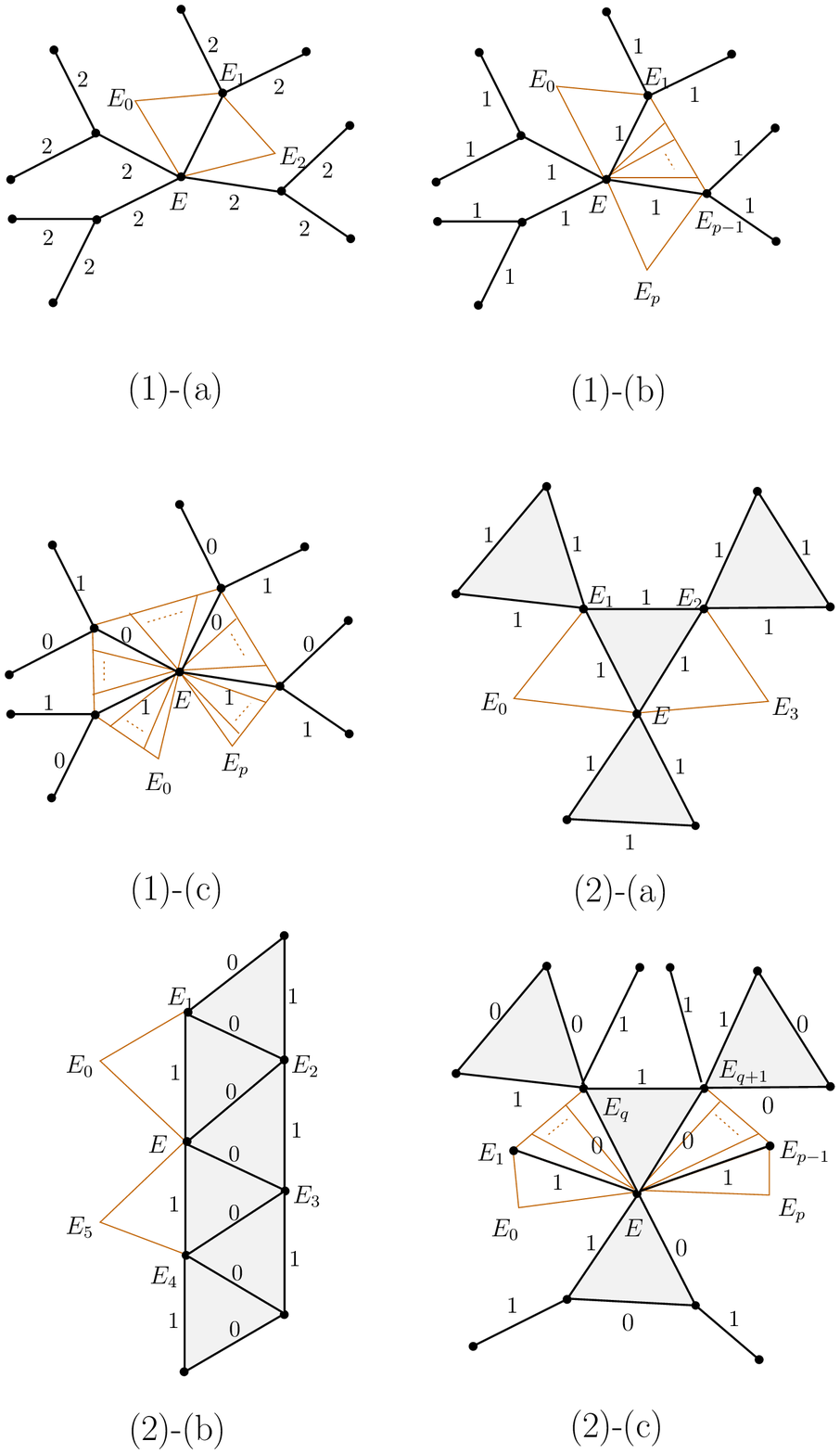}
\caption{A portion of each primitive disk complex $\mathcal P(V)$ together with simplices whose vertices are represented by disks in a sequence generated by a dual pair. Each number designates the type of the edge.}
\label{shape}
\end{figure}

\smallskip
\noindent {\bf Acknowledgments.}
The authors wish to express their gratitude to Darryl McCullough for helpful discussions with his valuable advice and comments.
Part of this work was carried out while the second-named author
was visiting Hanyang University.
Grateful acknowledgment is made for hospitality.

\bibliographystyle{amsplain}

\end{document}